\documentclass{amsart}
\usepackage{amsmath,amssymb,dsfont}
\usepackage{graphicx,color}
\usepackage{pstricks, pst-plot, pst-node, pst-grad, pst-math}
\usepackage{pst-plot}
\usepackage{hyperref}
\usepackage{todonotes}
\usepackage{comment}
\newtheorem{theorem}{Theorem}[section]
\newtheorem{lemma}[theorem]{Lemma}
\newtheorem{proposition}[theorem]{Proposition}
\newtheorem{corollary}[theorem]{Corollary}

\theoremstyle{definition}

\newtheorem{example}[theorem]{Example}
\newtheorem{remark}[theorem]{Remark}

\numberwithin{equation}{section}

\newcommand{\rd}{{\,\rm d}}
\newcommand{\e}{{\rm e}}

\newcommand{\R}{{\mathbb R}}
\newcommand{\C}{{\mathbb C}}

\newcommand\eps{\varepsilon}

\newcommand\beq{\begin{equation}}
\newcommand\eeq{\end{equation}}

\newcommand{\dist}{\mathrm{dist}}
\newcommand\re{\mathrm{Re}}
\newcommand\im{\mathrm{Im}}
\newcommand\I{\mathrm{i}}
\newcommand{\beqnt}{\begin{equation*}}
\newcommand{\eeqnt}{\end{equation*}}
\newcommand{\set}[2]{\{#1 : #2 \}}

\begin{document}
\author{Jean-Claude Cuenin}
\address{Mathematisches Institut, Ludwig-Maximilians-Universit\"at M\"unchen, 80333 Munich, Germany}
\email{cuenin@math.lmu.de}

\title[Schr\"odinger operators with slowly decaying potentials]{Improved eigenvalue bounds for Schr\"odinger operators with slowly decaying potentials}

\begin{abstract}
We extend a result of Davies and Nath \cite{MR1946184} on the location of eigenvalues of Schr\"odinger operators with slowly decaying complex-valued potentials to higher dimensions. In this context, we also discuss various examples related to the
Laptev--Safronov conjecture \cite{MR2540070}.
\end{abstract}

\maketitle

\section{Introduction and main result}

Eigenvalue estimates for Schr\"odinger operators $-\Delta+V$ on $L^2(\R^d)$ with complex-valued potentials $V\in L^q(\R^d)$ have been intensively studied over the past two decades by many authors, starting with the observation of Abramov, Aslanyan and Davies \cite{AAD01} that in one dimension the bound
\begin{align}\label{AAD}
|z|^{\frac{1}{2}}\leq \frac{1}{2}\int_{\R}|V(x)|\rd x
\end{align}
holds for any eigenvalue $z\in\C\setminus[0,\infty)$ of $-\Delta+V$. This scale-invariant bound has the same form (up to replacing $|V(x)|$ by $V(x)_-$) as the endpoint Lieb-Thirring inequality in $d=1$ (\cite{LiebThirring,MR1663336}) for a potential with a single eigenvalue. Higher-dimensional versions of \eqref{AAD} were proved by Frank \cite{MR2820160,MR3717979} and Frank--Simon \cite{MR3713021}. A maybe less well-known bound in the one-dimensional case, due to Davies and Nath \cite{MR1946184}, improves \eqref{AAD} to
\begin{align}\label{DN}
|z|^{\frac{1}{2}}\leq \frac{1}{2}\sup_{y\in\R}\int_{\R}|V(x)|\exp(-\im\sqrt{z}|x-y|)\rd x,
\end{align} 
which is valid even for slowly decaying potentials, i.e.\ $V\notin L^1(\R)$.  
The aim of this note is to prove the following higher-dimensional analogue of \eqref{DN}.

\begin{theorem}\label{thm DN higher dim}
Let $d\geq 2$, $q\in [\max(d/2,1+),(d+1)/2]$ and $V\in L^{q}_{\rm loc}(\R^d)$. Then there exists $C_{d,q}>0$ such that any eigenvalue $z\in\C$ of $-\Delta+V$ satisfies
\begin{align}\label{DN higher dim}
|z|^{q-\frac{d}{2}}\leq C_{d,q}\sup_{y\in\R^d}\int_{\R^d}|V(x)|^{q}\exp(-\im\sqrt{z}|x-y|)\rd x.
\end{align}
\end{theorem}
Here, $1+$ denotes an arbitrary number $>1$ (this is only relevant if $d=2$) and the branch of the square root is chosen such that $\im\sqrt{z}>0$.
In the ``short-range" case $V\in L^q(\R^d)$, $q\leq (d+1)/2$, Theorem \ref{thm DN higher dim} recovers the results of \cite{MR2820160}, simply by neglecting the exponential.
In the ``long-range" case $q>(d+1)/2$ an application of H\"older's inequality yields bounds that are close to those of \cite{MR3717979}; see Corollary \ref{cor 1} and the subsequent remark for details. We also mention the recent result of Lee and Seo \cite{MR3865141} where the $L^q$ norm is replaced by the Kerman--Saywer norm. An interesting consequence of our bound \eqref{DN higher dim} that cannot be deduced from those works is that, for a long-range potential $V$, we have the ``local bound"
\begin{align}\label{local bound}
|z|^{\frac{1}{2}}\leq C_d\sup_{y\in\R^d} \int_{\left(B\left(y,\frac{M}{\im\sqrt{z}}\right)\right)} |V(x)|^{\frac{d+1}{2}}\rd x,
\end{align}
where $M$ is some large number depending on $V$ and $z$ (see Corollary \ref{cor 2}). If more is known about the potential than just an $L^q$ norm, then \eqref{local bound} can yield considerably sharper bounds than those previously known in the literature. For example, we show that if $V$ is of ``Ionescu--Jerison" type in the sense discussed in \cite{MR3713021}, then the imaginary part of $z$ must decay exponentially fast as the ``size" of $V$ tends to zero (see Example \ref{example Ionescu--Jerison}). This observation may give a hint whether such a potential is a good candidate to disprove the open part of the so-called Laptev--Safronov conjecture \cite{MR2540070}, which stipulates that
\begin{align}\label{LS conjecture}
\sup_{V\in L^q(\R^d)}\sup_{z\in\sigma(-\Delta+V)\setminus\R_+}\frac{|z|^{q-\frac{d}{2}}}{\|V\|_q^q}<\infty\quad\mbox{for all}\quad q\in[d/2,d].
\end{align}
For the range $q\in[d/2,(d+1)/2]$ the conjecture was proven by Frank \cite{MR2820160}. The question whether \eqref{LS conjecture} is true for $q\in((d+1)/2,d]$ is still open. 
The lower bound for $q$ in \eqref{LS conjecture} is obvious and already appears for real-valued potentials. The conjectured upper bound $q\leq d$ was based on the observation that 
there are examples of potentials, due to Wigner and Von Neumann, that decay like $1/|x|$ at infinity and give rise to embedded eigenvalues. For these potentials $q=d$ would be borderline in terms of integrability. However, there are examples of embedded eigenvalues, due to Ionescu and Jerison \cite{MR2024415}, for non-radial potentials that are in $L^q(\R^d)$ for any $q>(d+1)/2$; see also \cite{MR3713021} for a simplified version of the Ionescu--Jerison example and \cite{arXiv:1709.06989} for additional examples as well as an explanation of the connection to the Knapp example in harmonic analysis. Thus, the expectation is 
that the conjecture is false in the range $q\in((d+1)/2,d]$. That \eqref{LS conjecture} cannot hold for $q>d$, as conjectured by Laptev and Safronov, was proved by B\"ogli \cite{MR3627408}.

Finally we should also mention that there are versions of \eqref{AAD} concerning sums of eigenvalues (e.g.\ \cite{MR2260376,MR2559715,MR3077277,MR3730931,2016arXiv160103122F}), but these  will not be discussed here. Several works also deal with a class of potentials outside the $L^q$-scale (e.g.\ \cite{MR3812810,MR3847475}), Schr\"odinger operators with inverse square potentials \cite{2016arXiv160701727M}, Schr\"odinger operators on conical manifolds with non-trapping metrics \cite{2017arXiv170909759G}, fractional Schr\"odinger and Dirac operators \cite{MR3608659}, to name just a few.

\section{Proof of Theorem \ref{thm DN higher dim}}

To ease notation we define
\begin{align*}
F_V^q(s):=\left(\sup_{y\in\R^d}\int_{\R^d}|V(x)|^{q}\exp(-s|x-y|)\rd x\right)^{\frac{1}{q}}.
\end{align*}
The first step is to reduce the proof of Theorem \ref{thm DN higher dim} to the Birman-Schwinger bound
\begin{align}\label{BS}
\||V|^{\frac{1}{2}}(-\Delta-z)^{-1}|V|^{\frac{1}{2}}\|\leq C_{d,q}|z|^{\frac{d}{2q}-1}F_V^q(\im\sqrt{z}).
\end{align}
We first prove that it is sufficient to establish \eqref{BS} for $|z|=1$, then reduce the proof to a pointwise bound. 
\begin{lemma}[Scaling]
If \eqref{BS} holds for $|z|=1$, then it holds for all $z\in\C$.
\end{lemma}

\begin{proof}
One easily checks that
\begin{align*}
\||V|^{\frac{1}{2}}(-\Delta-z)^{-1}|V|^{\frac{1}{2}}\|=|z|^{-1}\||V(\cdot/\sqrt{|z|})|^{\frac{1}{2}}(-\Delta-z/|z|)^{-1}|V(\cdot/\sqrt{|z|})|^{\frac{1}{2}}\|.
\end{align*}
Hence, if \eqref{BS} held for $|z|=1$ this would imply that
\begin{align*}
\||V|^{\frac{1}{2}}(-\Delta-z)^{-1}|V|^{\frac{1}{2}}\|\leq C_{d,q}|z|^{-1}F^q_{V(\cdot/\sqrt{|z|})}(\im\sqrt{z/|z|}).
\end{align*}
Since 
\begin{align*}
F^q_{V(\cdot/\sqrt{|z|})}(s)=|z|^{\frac{d}{2q}}F_V^q(s\sqrt{|z|}),
\end{align*}
the claim follows.
\end{proof}

\begin{lemma}[Pointwise bounds]
Let $z,\zeta\in\C$, with $|z|=1$, $\im \,z\neq 0$ and $\re\,\zeta\in [d/2,(d+1)/2]$. Then the kernel $K_{z,\zeta}$ of $(-\Delta-z)^{-\zeta}$ satisfies the bound
\begin{align}\label{pointwise bound}
|K_{z,\zeta}(x-y)|\leq C\e^{c|\im\zeta|}\e^{-\im\sqrt{z}|x-y|}|x-y|^{-\frac{d+1}{2}+\re\,\zeta}.
\end{align}
\end{lemma}

\begin{proof}
This follows from the explicit formula for the kernel $K_{z,\zeta}$ and standard Bessel function estimates, see e.g.\ (2.21)--(2.27) in \cite{MR894584} or the proof of (2.5) in the appendix of \cite{MR3865141} (where the estimate is in fact proved for the larger range $\re\,\zeta\in [(d-1)/2,(d+1)/2]$, but this will not be needed here). In both references the (second) exponential factor in \eqref{pointwise bound} is simply estimated by one. 
\end{proof}

We need the following simple version of Schur's test. The proof follows from a routine application of the Cauchy-Schwarz inequality and is omitted.

\begin{lemma}[Schur's test]
Let $\rho:\R^d\times \R^d\to (0,\infty)$. Suppose that $T$ is an operator on $L^2(\R^d)$ with locally integrable kernel $K$. Then
\begin{align*}
\|T\|_{L^2\to L^2}\leq \left(\sup_{x\in\R^d}\int_{\R^d}|K(x,y)|\rho(x,y)^{-1}\rd y\right)^{\frac{1}{2}}\left(\sup_{y\in\R^d}\int_{\R^d}|K(x,y)|\rho(x,y)\rd x\right)^{\frac{1}{2}}.
\end{align*} 
\end{lemma}

\begin{proof}[Proof of Theorem \ref{thm DN higher dim}]
We will apply Stein's complex interpolation theorem (see e.g.\ \cite{MR0304972} for a textbook reference) to the analytic family of operators
\begin{align*}
T_{\zeta}=|V|^{\frac{\zeta}{2}}(-\Delta-z)^{-\zeta}|V|^{\frac{\zeta}{2}},
\end{align*}
where $|z|=1$ and $0\leq \re\,\zeta\leq (d+1)/2$ (see also \cite{MR3730931,MR3608659,MR3717979,MR3865141} where the same family is considered). It suffices to prove the bounds
\begin{align*}
\|T_{\zeta}\|_{L^2\to L^2}\leq C\e^{c|\im\zeta|}\quad \mbox{for  } &\re\,\zeta=0,\\
\|T_{\zeta}\|_{L^2\to L^2}\leq C\e^{c|\im\zeta|}F_V^q(\im\sqrt{z})\quad \mbox{for  } &\re\,\zeta=q 
\end{align*}
The first bound immediately follows from Plancherel's theorem (see e.g.\ the proof of (2.3) in \cite{MR3865141}). The second bound follows from \eqref{pointwise bound} and Schur's test with 
\begin{align*}
\rho(x,y)=\frac{|V(x)|^{\frac{q}{2}}}{|V(y)|^{\frac{q}{2}}}.
\end{align*}
To be precise, we first truncate $|V|$ from above and below, so that $\rho$ and $\rho^{-1}$ are bounded. The truncation can be removed at the end.
\end{proof}

\section{Consequences of Theorem \ref{thm DN higher dim}}

\begin{corollary}\label{cor 1}
Let $d\geq 2$ and $V\in L^{q}(\R^d)$ for some $q\geq (d+1)/2$. Then there exists $C_{d,q}>0$ such that any eigenvalue $z\in\C$ of $-\Delta+V$ satisfies
\begin{align}\label{bound cor 1}
|z|^{\frac{1}{d+1}}(\im\sqrt{z})^{d\left(\frac{2}{d+1}-\frac{1}{q}\right)}\leq C_{d,q} \|V\|_q
\end{align}
\end{corollary}

\begin{proof}
This follows from \eqref{DN higher dim} with $q=(d+1)/2$ by H\"older's inequality.
\end{proof}

\begin{remark}
Since $\im\sqrt{z}\gtrsim \dist(z,\R_+)/\sqrt{|z|}$ it immediately follows from \eqref{bound cor 1} that
\begin{align}\label{bound cor of cor 1}
|z|^{\frac{1}{d+1}-\frac{d}{2}\left(\frac{2}{d+1}-\frac{1}{q}\right)}\dist(z,\R_+)^{d\left(\frac{2}{d+1}-\frac{1}{q}\right)}\leq C_{d,q} \|V\|_q.
\end{align}
This estimate is not so good for large $q$; for example, for $q=\infty$ the trivial bound
\begin{align}\label{trivial spectral bound}
\dist(z,\R_+)\leq \|V\|_{\infty}
\end{align}
easily beats \eqref{bound cor 1} since $\dist(z,\R_+)\leq |z|$. In terms of the Birman-Schwinger operator the inequality leading to \eqref{trivial spectral bound} is of course
\begin{align}\label{trivial BS bound}
\||V|^{\frac{1}{2}}(-\Delta-z)^{-1}|V|^{\frac{1}{2}}\|\leq \dist(z,\R_+)^{-1}\|V\|_{\infty}.
\end{align}
Since the right hand side of \eqref{BS} is clearly bounded by $|z|^{\frac{1}{d+1}}\|V\|_{d+1}$, complex interpolation between \eqref{BS} and \eqref{trivial BS bound} yields
\begin{align}\label{Ruperts bound}
|z|^{\frac{1}{2q}}\dist(z,\R_+)^{1-\frac{d+1}{2q}}\leq C_{d,q} \|V\|_q
\end{align}
for $q\geq (d+1)/2$. This bound was proved by Frank \cite{MR3717979}, and we refer to this paper for the details of the complex interpolation.
\end{remark}

\begin{corollary}\label{cor 2}
Let $d\geq 2$ and $V\in L^{q}(\R^d)$ for some $q> (d+1)/2$. If $z\in\C$ is an eigenvalue of $-\Delta+V$, then there exists $C_d'>0$ such that for any $M\geq 0$ satisfying
\begin{align}\label{asspt on M}
M\geq (d+1)\ln\|V\|_q-2d\beta_q^{-1}\ln(\beta_q\,\im\sqrt{z})-\ln|z|+(d+1)\ln (2C_d'),
\end{align}
where $\beta_q^{-1}=1-(d+1)/(2q)$, the bound
\begin{align*}
|z|^{\frac{1}{d+1}}\leq 2C_d'\sup_{y\in\R^d} \|V\|_{L^{\frac{d+1}{2}}\left(B\left(y,\frac{M}{\im\sqrt{z}}\right)\right)}
\end{align*}
holds. 
\end{corollary}

\begin{proof}
We split the integral in \eqref{DN higher dim} (again with $q=(d+1)/2$) into a region $|x-y|\leq M/\im\sqrt{z}$ and its complement. Estimating the exponential factor by $1$ in the first region and using H\"older in the second yields  
\begin{align*}
|z|^{\frac{1}{d+1}}\leq C_d'
\left(\sup_{y\in\R^d} \|V\|_{L^{\frac{d+1}{2}}\left(B\left(y,\frac{M}{\im\sqrt{z}}\right)\right)}+\e^{-\frac{M}{d+1}}(\beta_q\,\im\sqrt{z})^{-d\left(\frac{2}{d+1}-\frac{1}{q}\right)}\|V\|_q
\right)
\end{align*}
for some constant $C_d'>0$ that is a multiple of $C_{d,\frac{d+1}{2}}$ in \eqref{DN higher dim}. By the choice of $M$ in \eqref{asspt on M} the second term is at most half the size of the left hand side and can thus be absorbed.
\end{proof}

\section{A sharp bound for quasimodes}

In this section we prove a statement that is slightly stronger than that in \cite{MR3717979}. Example~\ref{example sharp quasimode} below shows that this stronger version is sharp. Since \eqref{DN higher dim} is scale-invariant, we may assume that $|z|=1$ in the following. We consider the following generalized eigenvalue or quasimode equation,
\begin{align}\label{quasimode equation}
(-\Delta+V_n-z_n)\psi_n=g_n,
\end{align}
where $g_n$ is a suitably small error, made precise in \eqref{quasimode condition} below.

\begin{proposition}\label{prop. quasimode}
Let $\epsilon_n$ be a sequence of positive numbers tending to zero, and let $z_n$ be a sequence of complex numbers with $|z_n|\approx 1$ and $\im\,z_n=\epsilon_n$. Assume that there exist functions $\psi_n\in H^2(\R^d)$, $g_n\in L^2(\R^d)$ and $V_n\in L^q(\R^d)$, for some $q>(d+1)/2$, such that \eqref{quasimode equation} holds. Then there exists $\delta>0$ such that if\footnote{Note that $\|V_n^{\frac{1}{2}}\psi_n\|_2$ is finite in view of Sobolev embedding.} 
\begin{align}\label{quasimode condition}
\limsup_{n\to\infty}\epsilon_n^{\frac{d+1}{4q}-1}\|V_n\|_q^{\frac{1}{2}}\frac{\|g_n\|_2}{\|V_n^{\frac{1}{2}}\psi_n\|_2}\leq \delta,
\end{align}
then the bound
\begin{align}\label{quasimode bound}
\epsilon_n^{1-\frac{d+1}{2q}}\leq C_{d,q} \|V_n\|_q
\end{align}
holds for sufficiently large $n$.
\end{proposition}

\begin{remark}
If $g_n=0$, then condition \eqref{quasimode condition} is void, and we recover the eigenvalue bound \eqref{Ruperts bound} for fixed $n$.
\end{remark}

For the proof of Proposition \ref{prop. quasimode} we will need the following consequence of the Stein--Tomas theorem.

\begin{lemma}\label{lemma 2-pc}
Let $n\geq 2$. Then, for $\lambda^{-1}\leq \epsilon\leq 1$, we have 
\begin{align}\label{2-pc}
\|(-\Delta-(\lambda+\I\epsilon)^2)^{-1}\|_{L^2\to L^{p_c}}\lesssim \epsilon^{-\frac{1}{2}}\lambda^{\frac{1}{p_c}-1}\quad (p_c=2(d+1)/(d-1)).
\end{align}
\end{lemma}

\begin{proof}
The proof is essentially contained in \cite{MR894584}. For completeness, we provide a full proof here (communicated to the author by C.D.\ Sogge). By duality, and due to the assumptions on $\epsilon,\lambda$, it is sufficient to prove that
\begin{align}\label{pc'-2}
\left\|\int\frac{\e^{\I x\cdot\xi}\widehat{f}(\xi)}{|\xi|^2-\lambda^2-\I\epsilon\lambda}\rd \xi\right\|_{2}\lesssim  \epsilon^{-\frac{1}{2}}\lambda^{\frac{1}{p_c}-1}\|f\|_{p_c'}.
\end{align}
The Stein--Tomas theorem asserts that
\begin{align}\label{Stein--Tomas}
\left(\int_{S^{d-1}}|\widehat{h}(\omega)|^2\rd \omega\right)^{\frac{1}{2}}\lesssim  \|h\|_{p_c'}.
\end{align} 
By scaling, \eqref{Stein--Tomas} is equivalent to
\begin{align}\label{Stein--Tomas scaling}
\left(\int_{S^{d-1}}|\widehat{h}(r\omega)|^2r^{d-1}\rd \omega\right)^{\frac{1}{2}}\lesssim r^{\frac{1}{p_c}} \|h\|_{p_c'}.
\end{align} 
Using polar coordinates, Plancherel's theorem and \eqref{Stein--Tomas scaling}, we get, since $p_c>2$,
\begin{align*}
\left\|\int\frac{\e^{\I x\cdot\xi}\widehat{f}(\xi)}{|\xi|^2-\lambda^2-\I\epsilon\lambda}\rd \xi\right\|_{2}^2
&\lesssim\lambda^{-2}\int_0^{\infty}\int_{S^{d-1}}|\widehat{f}(r\omega)|^2r^{d-1}\rd \omega\frac{\rd r}{(r-\lambda)^2+\epsilon^2}\\
&\lesssim\lambda^{-2}\left(\int_0^{\infty}\frac{r^{\frac{2}{p_c}}}{(r-\lambda)^2+\epsilon^2}\rd r\right)\|f\|^2_{p_c'}\\
&\lesssim \epsilon^{-1}\lambda^{\frac{2}{p_c}-2}\|f\|^2_{p_c'}.
\end{align*}
\end{proof}

\begin{corollary}
For $q>(d+1)/2$, $|z|\approx 1$, $|\im \,z|=\epsilon\ll 1$, we have
\begin{align}\label{fg}
\|f(-\Delta-z)^{-1}g\|\lesssim \epsilon^{-1+\frac{d+1}{4q}}\|g\|_2\|f\|_{2q}.
\end{align}
\end{corollary}

\begin{proof}
Without loss of generality we may assume that $f,g>0$. We apply Stein's interpolation theorem to the analytic family
\begin{align*}
T_{\zeta}=f^{\zeta}(-\Delta-z)^{-1},\quad 0\leq\re\,\zeta\leq \frac{2q}{d+1}.
\end{align*}
It is sufficient to prove the two estimates
\begin{align*}
\|T_{\zeta}\|\leq \epsilon^{-1},\quad \re\,\zeta&=0,\\
\|T_{\zeta}\|\lesssim \epsilon^{-\frac{1}{2}}\|f\|^{\frac{q}{d+1}}_q,\quad \re\,\zeta &=\frac{2q}{d+1}.
\end{align*}
The first is just the trivial bound \eqref{trivial BS bound}. The second follows from \eqref{2-pc} and H\"older.
\end{proof}

\begin{proof}[Proof of Proposition \ref{prop. quasimode}]
From the quasimode equation \eqref{quasimode equation} it follows that
\begin{align*}
|V_n|^{\frac{1}{2}}\psi_n=|V_n|^{\frac{1}{2}}(-\Delta-z_n)^{-1}g_n-|V_n|^{\frac{1}{2}}(-\Delta-z_n)^{-1}V_n^{\frac{1}{2}}(|V_n^{\frac{1}{2}}|\psi_n).
\end{align*}
By \eqref{fg} this implies that
\begin{align*}
1\lesssim \|V_n^{\frac{1}{2}}(-\Delta-z_n)^{-1}V_n^{\frac{1}{2}}\|+\epsilon^{-1+\frac{d+1}{4q}}\|V_n\|_{q}^{\frac{1}{2}}\frac{\|g_n\|_2}{\|V_n^{\frac{1}{2}}\psi_n\|_2}.
\end{align*}
By assumption \eqref{quasimode condition} we can absorb the second term. Then the usual Birman--Schwinger argument applies and yields \eqref{quasimode bound}, in parallel to the poof of \eqref{Ruperts bound}.
\end{proof}

\section{Examples}

\begin{example}
Fix $q>(d+1)/2$, $\mu\in (0,1]$, and consider the ball 
\begin{align*}
\mathcal{B}_{q,\mu}:=\set{V\in L^{q}(\R^d)}{\|V\|_q\leq \mu}.
\end{align*}
Also fix a small number $\delta>0$, and consider the rectangle 
\begin{align*}
\Omega_{\delta}:=\set{z\in\C}{1\leq \re\, z\leq 2,\, 0<\im\, z\leq \delta}
\end{align*}
in the upper half plane (the horizontal position is not so important, only boundedness of $\re\, z$ from above and below, i.e.\ away from zero, is needed). Note that $\im\sqrt{z}=(\im\,z)/2+\mathcal{O}(\delta^2)$ for $z\in\Omega_{\delta}$, as $\delta$ tends to zero. Assume that $z\in\Omega_{\delta}$ is an eigenvalue of $-\Delta+V$, where $V\in \mathcal{B}_{q,\mu}$. Let $\delta$ be so small that the last term in \eqref{asspt on M} is bounded from above by $-d\beta_q^{-1}\ln(\beta_q\,\delta)$. Dropping negative terms in \eqref{asspt on M} we may choose $M:=-3d\beta_q^{-1}\ln(\beta_q\,\im\,z)$.
Corollary \ref{cor 2} then implies the lower bound
\begin{align}\label{lower bound}
\sup_{y\in\R^d} \|V\|_{L^{\frac{d+1}{2}}\left(B\left(y,6d\frac{|\ln(\beta_q\,\im\,z)|}{\beta_q\,\im\,z}\right)\right)}\geq c>0.
\end{align}
In other words, a potential $V\in \mathcal{B}_{q,\mu}$ giving rise to an eigenvalue $z\in\Omega_{\delta}$ must have positive $L^{\frac{d+1}{2}}$ mass over some ball of radius $|\ln \im\, z|/\im\, z$.
\end{example}

\begin{example}\label{example Ionescu--Jerison}
This is a continuation of the previous example. We consider a family of potentials $V_n$, depending on a large parameter $n$ and satisfying the upper bound
\begin{align}\label{Ionescu-Jerison}
|V_n(x)|\lesssim (n+|x_1|+|x'|^2)^{-1}.
\end{align}
Such potentials naturally appear in examples of absence of embedded eigenvalues, see Ionescu--Jerison \cite{MR2024415}, Frank--Simon \cite{MR3713021} and the author \cite{arXiv:1709.06989}. Denote 
\begin{align}\label{U}
U_{n,\kappa}:=V_n+\kappa W,
\end{align}
where $W\in\mathcal{S}(\R^d)$ is a fixed potential, $V_n$ satisfy \eqref{Ionescu-Jerison}, and $|\kappa|\ll 1$ is a small parameter. In \cite{MR3713021} it is mentioned that $U_{n,\kappa}$ would be a plausible candidate to disprove the so-called Laptev--Safronov conjecture. This would be achieved if one could show that there is a sequence of eigenvalues $z_{n,\kappa}$ of $-\Delta+U_{n,\kappa}$ such that
\begin{align}\label{LS disprove}
\lim_{(n,\kappa)\to(\infty,0)}\frac{|z_{n,\kappa}|^{q-\frac{d}{2}}}{\|U_{n,\kappa}\|_q^q}=\infty
\end{align} 
for every $q>(d+1)/2$. For a sequence $z_{n,\kappa}\in\Omega_{\delta}$ (as in the previous example) this is equivalent to
\begin{align}\label{LS disprove simple}
\lim_{(n,\kappa)\to(\infty,0)}\|U_{n,\kappa}\|_q^q=0.
\end{align}
It is easy to check that $V_n,U_{n,\kappa}\in L^q(\R^d)$ for any $q>(d+1)/2$, with
\begin{align}\label{Lq norms of Vn and Unk}
\|V_n\|_q=\mathcal{O}\left(n^{\frac{d+1}{2q}-1}\right),\quad
\|U_{n,\kappa}\|_q=\mathcal{O}\left(n^{\frac{d+1}{2q}-1}+|\kappa|\right),
\end{align}
and hence \eqref{LS disprove simple} holds.
We now show that a necessary condition for $z_{n,\kappa}$ to be an eigenvalue of $-\Delta+U_{n,\kappa}$ is that
\begin{align}\label{nec. condition LS}
|\im\,z_{n,\kappa}|\leq C\e^{-cn}
\end{align}
for some constants $C,c>0$ and for $n$ large and $|\kappa|$ small enough. In particular, this implies that the bound \eqref{Ruperts bound} is not saturated for the potentials $U_{n,\kappa}$, and that Corollary \ref{cor 2} yields much better bounds in this case.
Recall that we assume $\im\,z_{n,\kappa}>0$, but the same argument works for $\im\,z_{n,\kappa}<0$.
The first observation is that in view of \eqref{Ruperts bound} (or even \eqref{bound cor of cor 1}) the condition \eqref{LS disprove simple} implies
\begin{align*}
\lim_{(n,\kappa)\to(\infty,0)}|\im\,z_{n,\kappa}|=0.
\end{align*}
Hence, we may choose $\delta$ in the definition of $\Omega_{\delta}$ arbitrarily small. Since $U_{n,\kappa}\in L^q(\R^d)$ for any $q>(d+1)/2$ we may fix such a $q$ and a corresponding $\beta_q$ as in Corollary \ref{cor 2}. We first prove \eqref{nec. condition LS} in the case $\kappa=0$. Since we are only concerned with norms we may substitute $V_n$ by the right hand side of \eqref{Ionescu-Jerison} for the following argument. We may then write the bound \eqref{lower bound} as 
\begin{align}\label{lower bound A}
\|V_n\|_{L^{\frac{d+1}{2}}(B(0,R))}\gtrsim 1,\quad \mbox{where}\quad
R:=A\frac{|\ln\epsilon|}{\epsilon},\quad \epsilon=\im\,z_{n,\kappa},
\end{align}
and where $A$ is some sufficiently large constant. We argue by contradiction. Assume that \eqref{nec. condition LS} failed, i.e.\ that for any $C,c>0$ there are $n$ and $\kappa$ such that
\begin{align}\label{contradiction assumption}
\frac{\epsilon\e^{cn}}{|\ln\epsilon|}\geq C.
\end{align}
Note that the logarithmic term can be bounded by an arbitrary power of $\epsilon^{-1}$ and can thus be absorbed into the constants by making $c$ slightly smaller and $C$ slightly larger. A straightforward computation shows that
\begin{align}\label{local norm of Vn}
\|V_n\|_{L^{\frac{d+1}{2}}\left(B\left(0,R\right)\right)}\lesssim\frac{1}{n}\max\left(1,\ln\left(\frac{R}{n}\right)\right).
\end{align}
If the maximum were in fact $1$, then \eqref{lower bound A} would imply that $n\lesssim 1$, which is absurd; hence we may replace the maximum by $\ln\left(\frac{R}{n}\right)$. Under assumption \ref{contradiction assumption} we have 
\begin{align*}
\frac{R}{n}\leq \frac{A}{C}\frac{\e^{cn}}{n}.
\end{align*}
Plugging this into \eqref{local norm of Vn} yields
\begin{align}\label{end of contradiction proof}
\|V_n\|_{L^{\frac{d+1}{2}}\left(B\left(0,R\right)\right)}\lesssim c
\end{align}
for $n$ sufficiently large. Since $c$ was arbitrarily small, this contradicts \eqref{lower bound A}. The proof for $\kappa\neq 0$ and $|\kappa|$ sufficiently small is an easy modification of the previous argument. The only difference is that $|\kappa|$ is added to the right hand sides of \eqref{local norm of Vn} and \eqref{end of contradiction proof}. 
\end{example}

\begin{example}\label{example rectangular well}
The next example is more informal than the previous ones. We consider the ``rectangular well"
\begin{align}\label{rectangular well}
V:=\alpha\,\mathbf{1}_{\mathbf{R}},\quad 
\textbf{R}:=[-R,R]\times [-\sqrt{R},\sqrt{R}]
\times\ldots\times [-\sqrt{R},\sqrt{R}].
\end{align}
where $\alpha\in\C$, $|\alpha|\ll 1$, and $R\gg 1$ are two parameters. We easily compute
\begin{align}\label{norm rectangular well}
\|V\|_q\approx |\alpha|R^{\frac{d+1}{2q}},\quad 1\leq q\leq \infty.
\end{align}
Suppose that $z$ is an eigenvalue of $-\Delta+V$. We omit the dependence of $V$ and $z$ on the parameters $\alpha$, $r$ in the notation. We denote $\epsilon=\im\, z$ and assume that $\epsilon>0$. Since $V$ has support in $\textbf{R}$, an eigenfunction $u$ (say $L^2$ normalized) corresponding to the eigenvalue $z$ decays exponentially outside of $\textbf{R}$, at a rate $\exp(-\epsilon|x|)$; this follows from inspection of the fundamental solution of $-\Delta-z$. If we multiply the eigenvalue equation
$-\Delta u+V u=z u$ by $\mathbf{1}_{\textbf{R}} \overline{u}$ and integrate by parts, we get informally
\begin{align}\label{IBP}
\int_{\textbf{R}}|\nabla u|^2+\alpha \int_{\textbf{R}} |u|^2=z \int_{\textbf{R}} |u|^2 +\mbox{(boundary terms)}.
\end{align} 
To make sense of the boundary terms we may slightly smooth out the rectangle $\textbf{R}$. In any event, due to the exponential decay of $u$ we have the rough bound
\begin{align}\label{boundary terms}
\mbox{(boundary terms)}\lesssim \exp(-c\epsilon \sqrt{R})\times \mbox{(perimeter of $\textbf{R}$)}.
\end{align}
Observe that the worst contribution of the boundary terms comes from those boundary surfaces that involve the long side of the rectangle. On the ``good" boundary surfaces one has the better exponential decay $\exp(-c\epsilon R)$. As we have to contend with \eqref{boundary terms}, we can neglect boundary terms provided that
\begin{align}\label{neglect boundary terms}
\sqrt{R}\geq \frac{C}{\epsilon}\left(\ln R+|\ln\epsilon|\right).
\end{align}
for some sufficiently large constant $C$. It is straightforward to see that if \eqref{neglect boundary terms} is satisfied when $|\ln\epsilon|$ is dropped from the right hand side, then it holds as stated for sufficiently large $R$. We thus assume that  
\begin{align}\label{relationship r epsilon bad}
\sqrt{R}/\ln R\geq C/\epsilon,
\end{align}
which also implies that $\sqrt{R}\geq C|\ln\epsilon|/\epsilon$.
Taking imaginary parts in \eqref{IBP} then yields
\begin{align*}
(\im\,\alpha-\epsilon)\int_{\textbf{R}} |u|^2=\mathcal{O}(\epsilon^2).
\end{align*}
One concludes (still informally) that $\im\,\alpha=\epsilon$ up to a small error that goes to zero as $R\to\infty$ and $\epsilon\to 0$.
Ignoring the $\ln R$ term in \eqref{relationship r epsilon bad} (since $q>(d+1)/2$ is fixed we always have an epsilon of room that allows us to replace $\ln R$ by a small power of $R$) and plugging this into \eqref{norm rectangular well} yields that $\|V\|_q\gtrsim \epsilon^{1-\frac{d+1}{q}}$; the right hand side only tends to zero if $q>d+1$. The loss of the ``$2$" in the denominator comes from the bad boundary estimate \eqref{boundary terms}. Had we only considered the good boundary surfaces we would have the better bound $\|V\|_q\gtrsim \epsilon^{1-\frac{d+1}{2q}}$. To circumvent the problem we simply \emph{assume} that 
\begin{align}\label{assumption r eps rectangle}
R=A\frac{|\ln\epsilon|}{\epsilon},\quad \alpha:=\I\,\epsilon.
\end{align}
We also (still) assume that $-\Delta+V$ has an eigenvalue $z$ with $|z|\approx 1$, $\im\, z\approx\epsilon$. From \eqref{norm rectangular well} and \eqref{assumption r eps rectangle} it is then easy to see that the bounds \eqref{Ruperts bound} and \eqref{lower bound A} are saturated up to logarithms. 
\end{example}

\begin{example}
We now establish a rigorous version of the last example in one dimension. The claim we are going to prove is the following: Given $\epsilon>0$ sufficiently small, there exists a ``complex square well potential" $V=V_0\mathbf{1}_Q$, $Q=[-R,R]$, $R\approx |\ln\epsilon|/\epsilon$, $V_0\in\C$, $|V_0|\approx\epsilon$, such that $-\partial^2+V$ has eigenvalue $(1+\I\epsilon)^2$ and $\|V\|_1\approx |\ln\epsilon|$. In particular, the bound \eqref{lower bound A} is saturated up to logarithms. 

\begin{proof}
Since $V(x)=V(-x)$ the wavefunction $\psi$ must be either even or odd.
We consider the even case. Then the Ansatz for $\psi$ is
\begin{align*}
\psi(x)=\begin{cases}
A\e^{-\I zx}\quad &(x\leq -R),\\
B\cos(\I kx)\quad &(-R\leq x\leq R),\\
A\e^{\I zx}\quad &(x\geq R),
\end{cases}
\end{align*}
where 
\begin{align}
\label{V0 vs k}
k^2=z^2-V_0
\end{align}
Continuity of $\psi,\psi'$ at $R$ is equivalent to
\begin{align*}
\det\begin{pmatrix}
\e^{\I zR}&-\cos(kR)\\
\I z\e^{\I zR}&k\sin(kR)
\end{pmatrix}
=0 \iff z=\I k\frac{\sin(kR)}{\cos(kR)}.
\end{align*}
This can be written as
\begin{align}\label{even 1}
z=-k\frac{1-\e^{2\I kR}}{1+\e^{2\I kR}}.
\end{align}
We make the change of variables
\begin{align*}
\omega=1+k\in \C, \quad z=1+\I\eps\in\C,\quad 0<\eps<1/2.
\end{align*}
For fixed $0<\rho<1$ define a function, depending on the parameters $R>0$, $\eps>0$, 
\begin{align}\label{even 2}
B(\I\eps,\rho \eps)\ni\omega\mapsto f_{\eps,R}(\omega):=\I\eps+\omega \frac{1-\e^{2\I(\omega-1)R}}{1+\e^{2\I(\omega-1)R}}.
\end{align}
Then it can be seen that \eqref{even 1} together with the condition $|1+k-\I\eps|<\rho\eps$ is equivalent to $f_{\epsilon,R}(\omega)=0$. 
We make the assumption that
\begin{align}\label{even 3}
(1-\rho)\eps R\geq -C-\ln\eps
\end{align}
for some large but fixed constant $0<C<-\ln\eps$,
which ensures that 
\begin{align}\label{even 4}
\sup_{\omega\in B(\I\eps,\rho \eps)}|\e^{2\I(\omega-1)R}|=\mathcal{O}(\eps^2).
\end{align}
In particular, we have that
\begin{align*}
f_{\eps,R}(\omega)=\I\eps+\omega+\mathcal{O}(\eps^2).
\end{align*}
Applying Lemma 2.23 in \cite{DZ19} yields that for $\eps$ sufficiently small and $R$ sufficiently large (depending on $\rho$) so that \eqref{even 3} holds, the function $f_{\epsilon,R}$ has exactly one simple zero in $B(\I\eps,\rho \eps)$, given by
\begin{align*}
\omega=-\I\eps(1+\mathcal{O}(\eps)).
\end{align*}
This means that \eqref{even 1} has a solution $k\in\C$ with $|1+k-\I\eps|<\rho\eps$ and therefore $z^2=(1+\I\eps)^2$ is an eigenvalue for the Schr\"odinger operator $-\partial^2+V$ with potential $V=V_0\mathbf{1}_{[-R,R]}$. Recalling \eqref{V0 vs k} we get the estimate
\begin{align*}
\|V\|_q
=|z^2-k^2|R^{1/q}=|z^2-(\omega-1)^2|R^{1/q}\leq 2\eps(2+\rho+\mathcal{O}(\eps)) R^{1/q}.
\end{align*}
Fixing $R$ by requiring equality in \eqref{even 3}, we get
\begin{align}\label{even 5}
\|V\|_q\approx\eps^{1-1/q}|\ln\eps|^{1/q},\quad 1\leq q\leq \infty.
\end{align}
\end{proof}
\end{example}

\begin{example}
The next example is a simplified version of \cite{MR3627408} in $d=3$ dimensions.\footnote{Similarly, the previous example could also be regarded as a simplified one-dimensional version.} Concretely, we show: Given $\epsilon>0$ sufficiently small, there exists a radial potential supported in $B(0,R)$, $R\approx |\ln\epsilon|/\epsilon$, such that $-\Delta^2+V$ has eigenvalue $(1+\I\epsilon)^2$ and $\|V\|_q\lesssim \epsilon^{1-3/q}$ for $q>3$ (up to logarithms).  
The potential in this example is too large to saturate the bound \eqref{lower bound A}; in fact, 
\begin{align*}
\|V\|_{L^{2}(B(0,R)}\approx \epsilon^{-\frac{1}{2}}\quad \mbox{(up to logaraithms)}.
\end{align*}
\begin{proof}
Let $z_1,z_2\in \C^+$ (i.e. $\im z_j>0$), $A_1,A_2\in\C$ and $R>0$. All parameters will be determined later. Then set
\begin{align*}
u(r):=\begin{cases}
A_1\frac{\sin(z_1r)}{r}\quad &(r\leq R),\\
A_2\frac{\e^{\I z_2r}}{r}\quad &(r\geq R).
\end{cases}
\end{align*}
In order to have $u\in H^2(\R^3)$ it is enough that $u$ and $u'$ are continuous at $r=R$. This is the case iff the linear equation
\begin{align*}
\begin{pmatrix}
\sin(z_1R)&-\e^{\I z_2R}\\z_1\cos(z_1R)&-\I z_2\e^{\I z_2R}
\end{pmatrix}
\begin{pmatrix}
A_1\\A_2
\end{pmatrix}
=\begin{pmatrix}
0\\ 0
\end{pmatrix}
\end{align*}
has a nontrivial solution $(A_1,A_2)^T$, hence iff
\begin{align*}
\det\begin{pmatrix}\sin(z_1R)&-\e^{\I z_2R}\\z_1\cos(z_1R)&-\I z_2\e^{\I z_2R}
\end{pmatrix}
=0
\iff  -\I z_2\sin(z_1R)+z_1\cos(z_1R)=0.
\end{align*}
In other words, by Euler's formula,
\begin{align}\label{z2 vs z1}
z_2=-z_1\frac{1+\e^{2\I z_1R}}{1-\e^{2\I z_2R}}.
\end{align}
We set $\re \,z_1=1/2$ and write $\im\, z_1=\eps>0$. Then
\begin{align*}
\im\, z_2&=-\eps-\sin(R)\e^{-2\eps R}+E,\\
|E|&\leq c_1 \eps^2+c_2\e^{-4\eps R}.
\end{align*} 
We set $R=C/(2\eps)$, with $C_0\leq C\leq -\ln((1+\delta)\eps)$. Here, $\delta>0$ is fixed and $C_0>0$ is such that $\e^{-C_0}+c_2\e^{-2C_0}\geq (1-\delta/4)\e^{-C_0}$. Let $\eps_0>0$ be such that $-\eps+c_1\eps^2\geq -(1+\delta/4)\eps$ for all $\eps\leq \eps_0$. We may also arrange that $\sin(R)=-1$. Then 
\begin{align*}
\im\, z_2&\geq -(1+\delta/4)\eps+(1-\delta/4)\e^{-C}\\
&\geq -(1+\delta/4)+(1-\delta/4)(1+\delta)\eps=(\delta/2+\mathcal{O}(\delta^2))\eps.
\end{align*}
Hence, there is $\delta_0>0$ such that for all $\delta\leq\delta_0$
\begin{align*}
\im\, z_2\geq \frac{\delta}{4}\eps.
\end{align*}
Then $u$ satisfies
\begin{align*}
|(\Delta+z_2^2)u(r)|\lesssim \e^{-C}\mathbf{1}\{r\leq C/(2\eps)\}|u(r)|.
\end{align*}
Hence, if $V$ is defined by the equation $(\Delta+z_2^2-V(r))u(r)=0$, we have
\begin{align*}
\|V\|_q\lesssim \e^{-C}(C/\eps)^{3/q}.
\end{align*}
The minimum is achieved for $C=-\ln((1+\delta)\eps)$, namely
\begin{align*}
\|V\|_q\lesssim \eps^{1-3/q} \quad\mbox{up to }\ln(\eps).
\end{align*}
This tends to $0$ as $\eps\to 0$, provided $q>d$. Note that since $\e^{\I z_2 r}/r$ is a fundamental solution to $-\Delta-z_2^2$, the support of $V$ is contained in $B(0,R)$.
\end{proof}
\end{example}

\begin{example}\label{example sharp quasimode}
Here we prove that the result of Proposition \ref{prop. quasimode} is sharp. We could take the rectangular well potential as in Example \ref{example rectangular well}, but for the sake of variety we consider a Gaussian potential $G(t)=\exp(-t^2/2)$. By slight abuse of notation we understand that $\epsilon>0$ is a sequence tending to zero. We suppress the dependence of $V,\psi,g$ in Proposition \ref{prop. quasimode} on the index $n$ of this sequence. We start with the quasimode
\begin{align*}
\psi(x):=N^{-1/2}\e^{\I x_1} G(|y|)_{y=(\epsilon x_1,\sqrt{\epsilon}x')},
\end{align*}
where $x=(x_1,x')\in\R\times\R^{d-1}$ and $N=\epsilon^{-\frac{d+1}{2}}$ is a normalization factor. We compute
\begin{align}\label{computation}
(-\Delta-1-\I\epsilon)\psi(x)=\left(\epsilon\left(d-1-|y'|^2+2\I y_1-\I\right)+\epsilon^2\left(1-y_1^2\right)\right)\psi(x),
\end{align}
where, as before, ${y=(\epsilon x_1,\sqrt{\epsilon}x')}$.
Hence, if $V(x)=\epsilon \chi(y)_{y=(\epsilon x_1,\sqrt{\epsilon}x')}$, where $\chi$ is an arbitrary Schwartz function, say with $\chi(0)=1$, we get
\begin{align*}
\|g\|_2:=\|(-\Delta-1-\I\epsilon+V)\psi\|_2\lesssim\epsilon,\quad \|V\|_q\approx 
\epsilon^{1-\frac{d+1}{2q}},\quad\|V^{\frac{1}{2}}\psi\|_2\approx 1.
\end{align*}
This implies that condition \eqref{quasimode condition} holds with a good margin and that \eqref{quasimode bound} is sharp.
\end{example}

\begin{example}
We show that the quasimode bound $\|g\|_2\lesssim \epsilon$ of the previous example may be improved (by changing the potential) to an exponentially small error. For this we set
\begin{align*}
V(x):=\epsilon V_1(x)+\epsilon^2 V_2(x)
\end{align*}
where
\begin{align*}
V_1(x)&:=-\left(d-1-|y'|^2+2\I y_1-\I\right)\mathbf{1}_{|y|\leq M},\\
V_2(x)&:=-\left(1-y_1^2\right)\mathbf{1}_{|y|\leq M},\\
\end{align*}
and $M\gg 1$ will be chosen later. Then \eqref{computation} yields
\begin{align*}
\|g\|_2:=\|(-\Delta-1-\I\epsilon+V)\psi\|_2
\lesssim \epsilon \exp\left(-\frac{1}{4} M^2\right).
\end{align*}
For fixed $(d+1)/2<q<\infty$ there exists $\delta\in (0,1)$ such that $q=(d+1)/(2(1-\delta))$. Then
\begin{align*}
\|V\|_q\lesssim 
\epsilon^{\delta}M^{2+\frac{d}{q}}.
\end{align*}
We now choose $M=\epsilon^{-\frac{\delta}{2(2+d/q)}}$, leading to
\begin{align*}
\|g\|_2\lesssim \epsilon \exp\left(-\frac{1}{4}\epsilon^{-\frac{\delta}{(2+d/q)}}\right),
\quad
\|V\|_q\lesssim \epsilon^{\frac{\delta}{2}}.
\end{align*}
\end{example}

\noindent\textbf{Acknowledgements:} The author gratefully acknowledges the hospitality of the Institut Mittag-Leffler and the invitation to the thematic program \emph{Spectral Methods in Mathematical Physics}. The present article was written during the authors stay. The idea to use Lemma \ref{lemma 2-pc} and the proof thereof is attributed to Chris Sogge, to whom the author is thankful.

\bibliographystyle{abbrv}

\end{document}